\documentclass[12pt]{amsart}
\usepackage{xypic}
\usepackage{amssymb}
\usepackage{amsthm}
\usepackage{bm}
\usepackage{latexsym}
\usepackage{amsmath}
\usepackage{eufrak}
\usepackage{mathrsfs}
\usepackage{amscd}
\usepackage[all]{xy}
\usepackage[usenames]{color}
\usepackage{graphicx}
\usepackage{epstopdf}
\usepackage{color}
\usepackage{amscd}
\theoremstyle{plain}

\newtheorem{theorem}{Theorem}[section]
\newtheorem{proposition}[theorem]{Proposition}
\newtheorem{lemma}[theorem]{Lemma}
\newtheorem{corollary}[theorem]{Corollary}

\theoremstyle{definition}

\newcommand{\appsection}[1]{\let\oldthesection\thesection
\renewcommand{\thesection}{Appendix \oldthesection}
\section{#1}\let\thesection\oldthesection}

\newtheorem{definition}[theorem]{Definition}

\theoremstyle{remark}

\newtheorem{remark}[theorem]{Remark}
\newtheorem{example}[theorem]{Example}

\def\F{{\mathbb{F}}}
\def\D{{\mathbb{D}}}

\def\Z{{\mathbb{Z}}}
\def\Q{{\mathbb{Q}}}
\def\C{{\mathbb{C}}}
\def\P{{\mathbb{P}}}

\def\O{{\mathcal{O}}}

\def\X{{\mathcal{X}}}
\def\Y{{\mathcal{Y}}}

\def\W{{\mathcal{W}}}

\def\DD{{\mathcal{D}}}

\newcommand{\eni}{mk1A}
\newcommand{\enii}{mk2A}

\begin{document}
\bibliographystyle{amsplain}
\title[Construcci\'on]{$\Q$-Gorenstein smoothings of surfaces and degenerations of curves}
\author{\textrm{Giancarlo Urz\'ua}}

\email{urzua@mat.puc.cl}

\maketitle

\begin{abstract}
In this paper we mainly describe $\Q$-Gorenstein smoothings of projective surfaces with only Wahl singularities which have birational fibers. For instance, these degenerations appear in normal degenerations of $\P^2$, and in boundary divisors of the KSBA compactification of the moduli space of surfaces of general type \cite{KSB88}. We give an explicit description of them as smooth deformations plus $3$-fold birational operations, through the flips and divisorial contractions in \cite{HTU13}. We interpret the continuous part (smooth deformations) as degenerations of certain curves in the general fiber. At the end, we work out examples happening in the KSBA boundary for invariants $K^2=1$, $p_g=0$, and $\pi_1=0$ using plane curves.
\end{abstract}

\section{Introduction} \label{intro}

$\Q$-Gorenstein smoothings are interesting degenerations of nonsingular surfaces. For instance, stable limits of nonsingular surfaces in the KSBA compactification of the moduli space of surfaces of general type \cite{KSB88} are $\Q$-Gorenstein smoothings. On the other hand, they provide a non-classical construction of nonsingular projective surfaces with fixed invariants by means of certain singular surfaces; this is the pioneering work of Y. Lee and J. Park \cite{LP07}. They are also present in degenerations of other type of surfaces \cite{Kaw92,Man91,HP10,Prok11}.

Let $\D$ be a smooth analytic germ of curve. We are interested in $\Q$-Gorenstein smoothings over $\D$ of irreducible projective surfaces $X$ with only quotient singularities. In this case the $\Q$-Gorenstein condition means that the canonical class of the corresponding $3$-fold is $\Q$-Cartier (see \cite{Hack11} for a general discussion). The singularities of $X$ are the so-called T-singularities \cite{KSB88}: they are either du Val singularities, or cyclic quotient singularities $\frac{1}{dn^2}(1,dna-1)$ with gcd$(n,a)=1$. For du Val singularities we have simultaneous resolution. For the others we have particular partial simultaneous resolution (given by the M-resolutions of \cite{BC94}) which has as special fiber a surface with only \emph{Wahl singularities}, this is, non du Val T-singularities with $d=1$ \cite[(5.9.1)]{Wahl81}. We remark that Wahl singularities are the log terminal singularities which have a rational homology disk smoothing. It turns out that $\Q$-Gorenstein smoothings over $\D$ of projective surfaces with only Wahl singularities have some common characteristics with nonsingular projective surfaces. In \S \ref{s1} we make this precise, showing their MMP with explicit birational operations as in \cite{HTU13}, minimal and canonical models, and some numerical invariants.

In general, it is not clear what produces that a nonsingular projective surface admits a $\Q$-Gorenstein degeneration with only Wahl singularities. This type of degenerations have no vanishing cycles. In \cite{Kaw92}, Kawamata describes $\Q$-Gorenstein degenerations when the general fiber has Kodaira dimension $0$ or $1$, giving a reason for the existence using elliptic fibrations. Hacking \cite{Hack11e} discusses it for surfaces of general type with geometric genus $0$ via exceptional vector bundles. We recall that the Kodaira dimension of the minimal resolution of the special fiber and of the general fiber may be different. For example, the special fiber could be rational and the general fiber could be of general type. This paper is mainly about $\Q$-Gorenstein smoothings over $\D$ of projective surfaces with only Wahl singularities so that the special and general fiber are birational. (It is actually more general, see Theorem \ref{birWsurf}.) These degenerations appear, for example, when studying normal degenerations of $\P^2$ \cite{Man91,HP10}, and when studying KSBA boundary in the moduli space of surfaces of general type with $p_g=0$ (cf. \cite{Urz13}). In this direction, Corollary \ref{birWsurfA} says: any birational $\Q$-Gorenstein smoothing comes from a smooth deformation (continuous part) followed by certain specific birational $3$-fold operations (discrete part). These operations are the explicit flips and divisorial contractions in \cite{HTU13}. A distinguished non-trivial case in this theorem is normal degenerations of $\P^2$, which is treated separately in \S \ref{s2} (see Example \ref{e1}).

In \S \ref{s3}, we look closer at the above ``continuous part". We interpret the deformation as degeneration of certain curves. The work with curves is possible because $\Q$-Gorenstein smoothings with only Wahl singularities happen on a $\Q$-factorial $3$-fold, and so we have intersection theory. This is a general property induced by normal surface singularities with a rational homology disk smoothing \cite[Prop.3.1]{Wahl11}. Using this we will keep track of curve degenerations after each flip or divisorial contraction; for the precise picture see \S \ref{s3}. In that section  we show conditions to produce many examples (Propositions \ref{obstr1} and \ref{obstr2}), and we give two concrete ones starting with $\P^2$ (Examples \ref{lines} and \ref{halphen}).

Our main motivation is to explore the KSBA boundary of the moduli space of simply connected surfaces of general type with $p_g=0$. The above continuous part gives something to work with when some boundary divisors, which parametrize rational singular surfaces, intersect. This continuous part gives degenerations of rational plane curves. Their parameters give the moduli for the divisors, and their further degenerations produce more divisors. In particular, this is an explicit description of the surfaces in any constructed Wahl divisor $\DD{n \choose a}$; cf. \cite{Urz13}. These plane curves may be irreducible, reducible but reduced, or reducible and nonreduced; some degenerations could give different families in the same moduli space (so the discrete part is now varying). In \S \ref{s4} we show four examples which give a first look at this way of describing KSBA boundary. All of them are simply connected, $K^2=1$, and $p_g=0$.

\subsection*{Notation} We use definitions, notations, and facts from \cite{HTU13}, \cite[Preliminaries]{Urz13}. As in \cite{Urz13}, the k1A and k2A extremal nbhds in \cite{HTU13} are denoted by \eni ~and \enii. Our ground field is $\C$.

\subsection*{Acknowledgements}

I have benefited from many conversations with Paul Hacking and Jenia Tevelev. I was supported by the FONDECYT Inicio grant 11110047 funded by the Chilean Government.

\tableofcontents

\section{W-surfaces and their MMP} \label{s1}

Let $X$ be a normal surface with only quotient singularities, and let $(0 \in \D)$ be a \emph{smooth analytic germ of curve}. A deformation $(X \subset \X) \rightarrow (0 \in \D)$ of $X$ is called a \emph{smoothing} if its general fiber is smooth. It is {\em $\Q$-Gorenstein} if $K_{\X}$ is $\Q$-Cartier. A germ of a normal surface $X$ is called a T-singularity if it is a quotient singularity and admits a $\Q$-Gorenstein smoothing. By \cite[Prop.3.10]{KSB88}, any T-singularity is either a du Val singularity or a cyclic quotient singularity of the form
${1\over dn^2}(1,dna-1)$ with gcd$(n,a)=1$. A T-singularity with a one-dimensional $\Q$-Gorenstein versal deformation space is either a node $A_1$ or a {\em Wahl
singularity} ${1\over n^2}(1,na-1)$. Hence Wahl singularities are precisely the T-singularities whose $\Q$-Gorenstein smoothing has Milnor number zero.

\begin{definition}
A {\em W-surface} is a normal projective surface $X$ together with a proper deformation $(X \subset \X) \to (0 \in \D)$ such that
\begin{enumerate}
\item $X$ has at most Wahl singularities.
\item $\X$ is a normal complex $3$-fold with $K_{\X}$ $\Q$-Cartier.
\item The fiber $X_0$ is reduced and isomorphic to $X$.
\item The fiber $X_t$ is nonsingular for $t\neq 0$.
\end{enumerate}
The W-surface is said to be {\em smooth} if $X$ is nonsingular.
\label{wsurf}
\end{definition}

These situation coincides with the moderate degenerations in \cite{Kaw92}. Notice that a W-surface has $\X$ terminal \cite[Cor.3.6]{KSB88} and $\Q$-factorial. The later assertion comes from the fact that $\X \to \D$ is a $\Q$-Gorenstein smoothing of singularities with Milnor number zero; see \cite[2.2.7]{Ko91}, or \cite[Prop.3.1]{Wahl11} in more generality. We use proper and not projective deformations since we will perform flips on such $\X$ which involve deformations of partial resolutions of singularities.

We point out that one produces many examples of W-surfaces in the following way; see \cite[(6.4)]{Wahl81}, \cite[\S1]{Man91}, \cite[\S2]{LP07}. One constructs a normal projective surface $X$ with only Wahl singularities and $H^2(X,T_X)=0$. Then local deformations glue to global deformations of $X$, and we consider the ones coming from $\Q$-Gorenstein smoothings of the singularities. This gives a W-surface. (For technical details in higher generality see \cite[\S3]{Hack11}.)

The aim of this section is to highlight similarities between W-surfaces and smooth projective surfaces via the general fiber. Let $X$ be a W-surface. Then, we have that $K_{X_t} = K_{\X}|_{X_t}$ for all $t$. Also the invariants $K_{X_t}^2$, $\chi_{\text{top}}(X_t)$, $q(X_t)=\text{dim}_{\C}H^1(X_t,\O_{X_t})$ (in general, see \cite{GS83}), and $p_g(X_t)=\text{dim}_{\C}H^2(X_t,\O_{X_t})$ are independent of $t$; cf. \cite{Kaw92}. A key property, which will be used frequently, is that the $3$-fold is $\Q$-factorial, which allows us to compare curves in the general and special fibers through intersection theory.

\begin{definition}
A W-surface $X$ is {\em minimal} if $K_X$ is nef. 
\end{definition}

\begin{lemma}
If a W-surface $X$ is minimal, then $K_{X_t}$ is nef for all $t$. 
\end{lemma}

\begin{proof}
Since $K_{X}$ is nef and $K_{X}^2=K_{X_t}^2$ for all $t$, we have $K_{X_t}^2 \geq 0$. Notice also that (by \cite[\S3]{GS83}) $q(X)=q(X_t)$ for all $t$, and so $b_1(X_t)$ is even since $X$ is projective with Wahl (and so rational) singularities. Suppose there is a $t\neq 0$ so that $K_{X_t}$ is not nef. Then since $K_{X_t}^2>0$ and $b_1(X_t)$ is even, we have that either $X_t$ is ruled or $X_t$ is a surface of general type \cite[p.91]{BHPV04}.

\textbf{(I)} Say that $X_t$ is ruled: Then $X_t$ must be rational since $K_{X_t}^2 \geq 0$. This is because a minimal model of $X_t$ has canonical class with self-intersection $8(1-q(X_t))$ and so $q(X_t)=1,0$. But if $q(X_t)=1$, then $X_t$ is a minimal minimal ruled surface since we must have $K_{X_t}^2=0$. Here, we can use \cite[\S3]{B86} to show that the minimal resolution $\widetilde{X}$ of $X$ must be a ruled surface as well. But then, since $q(\widetilde{X})=1$, we have that the exceptional divisor of the minimal resolution $\widetilde{X} \to X$ is contained in fibers of a fibration $\widetilde{X} \to E$, where $E$ is a curve of genus $1$. But then, the image of the general fiber of $\widetilde{X} \to E$ in $X$ is a $\P^1$ so that $\P^1 \cdot K_X=-2$, contradicting that $K_X$ is nef. Therefore, we are left with the case of a rational surface $X_t$. But then $q(X)=p_g(X)=0$, and so the deformation $(X \subset \X) \to (0 \in \D)$ is projective \cite[p.6]{Man91}. Now by \cite[Lemma(1.3)]{B86}, we have that $h^0(X,-mK_{X})\geq h^0(X_t,-mK_{X_t})$ for all $m>0$ divisible by the index of all Wahl singularities in $X$. But $h^0(X_t,-mK_{X_t})=h^1(X_t,-mK_{X_t}) + \frac{m(m+1)}{2}K_{X_t}^2 +1 >0$. On the other hand, $X$ is projective, and $K_X$ is nef, so $-mK_X$ cannot have global sections.

\textbf{(II)} Say that $X_t$ is of general type. Then one can prove the existence of a family of $(-1)$-curves $\Gamma_t$ degenerating to some effective divisor $\Gamma_0$ in $X$, which gives $K_{X_t} \cdot \Gamma_t=-1$ for $t \neq 0$, and so in the limit $K_{X} \cdot \Gamma_0 =-1$, which contradicts the fact that $K_X$ is nef. To see this family of $(-1)$-curves, we first notice that: if $X_{t_0}$ is of general type for some $t_0 \neq 0$, then all $X_t$ are of general type for $t \neq 0$ \cite[VI\S8]{BHPV04}. Now, by results of Kodaira and Iitaka \cite[IV\S4]{BHPV04}, there is a small disk $\D'$ in $\D \setminus \{ 0 \}$ around $t_0$ with a family of such $(-1)$-curves. We want to extend it to the whole disk. To prove this, it is enough to show that for a smooth deformation $(W \subset \W) \to (0 \in \D)$ of a smooth projective surface of general type $W$ there exists a smaller disk $\D'$ which contains $0$ such that all $(-1)$-curves in the neighbor fibers of $(W \subset \W') \to (0 \in \D')$ deform to a $(-1)$-curve in the central fiber $W$. First, to get this small disk $\D'$, we use several times the Kodaira and Iitaka results mentioned above to obtain a minimal model $W_m$ for $W$ together with a deformation $(W_m \subset \W_m) \to (0 \in \D')$. Notice that for general type surfaces, we have minimality if and only if $h^1(-K)=0$, and so the canonical class of every fiber in $(W_m \subset \W_m) \to (0 \in \D')$ is nef. Now restrict to the initial degeneration but over $\D'$: $(W \subset \W) \to (0 \in \D')$. Say that $\Gamma_t$ is a $(-1)$-curve in some fiber $W_t$ of $(W \subset \W) \to (0 \in \D')$ which does not belong to the exceptional divisor given by the contractions used to obtain $(W_m \subset \W_m) \to (0 \in \D')$. Then, its image ${\Gamma'}_t$ in the minimal model of $W_t$ (which is a fiber of $(W \subset \W) \to (0 \in \D')$) has ${\Gamma'}_t^2=-1+\sum m_i^2$ and arithmetic genus $p_a({\Gamma'}_t)=\sum \frac{m_i(m_i-1)}{2}$, for some integers $m_i \geq 0$, and so the intersection between ${\Gamma'}_t$ and the canonical divisor is negative by adjunction, a contradiction. Therefore $\Gamma_t$ belongs to the exceptional divisor, and this completes the proof.
\end{proof}

If $K_X$ is not nef, we run MMP \cite{KM1998} in the following way. First, following \cite[Thm.3.7]{KM1998} for example, there is a $K_X$-extremal rational curve $\Gamma$ (with $K_X \cdot \Gamma <0$). We have three options:

\vspace{0.3cm}
\textbf{(I)} If $\Gamma^2>0$, then Pic$(X)$ has rank $1$ and $-K_X$ is ample \cite[2.3.3]{KK}. Hence $-K_{X_t}$ is ample for any $t$ \cite[Prop.1.41]{KM1998}, and so $X_t$ is rational for any $t$. Moreover, the rank $1$ condition implies that $e(X_t)=3$ for all $t$, and so $X_t$ is isomorphic to $\P^2$. This type of degenerations of $\P^2$ were classified in \cite{Man91,HP10}. According to \cite[Cor.1.2]{HP10}, $X$ is a $\Q$-Gorenstein deformation of a weighted projective plane $\P(a,b,c)$ where $(a,b,c)$ satisfies the Markov equation $a^2+b^2+c^2=3abc$. In the next section we say more on those degenerations.

\vspace{0.3cm}
\textbf{(II)} If $\Gamma^2=0$, then there is a fibration $h \colon X \to B$ with irreducible fibers and general fiber isomorphic to $\P^1$ \cite[2.3.3]{KK}. Let $\tilde{h} \colon \tilde{X} \to T$ be the corresponding fibration on the minimal resolution $\tilde{X}$ of $X$. Then, over a $b \in B$ where the fiber has a Wahl singularity, the fiber in $\tilde{X}$ has two possible configuration types; see \cite[Prop.7.4]{HP10}. It is a simple check that none of them is possible when the singularities are of Wahl type. Therefore, $(X \subset \X) \to (0 \in \D)$ is a smooth deformation of a geometrically ruled surface $X$.

\vspace{0.3cm}
\textbf{(III)} If $\Gamma^2<0$, then we can apply to $(X \subset \X) \to (0 \in \D)$ a birational transformation defined by an extremal nbhd of type \eni ~or \enii ~of flipping or divisorial type; see \cite[Thm.5.3]{HTU13}. After that we arrive to another W-surface $(X^+ \subset \X^+) \to (0 \in \D)$. The fibers of $\X \to \D$ have changed in one of the following ways:

\textbf{(W-blow-down)} In the divisorial type, we have a birational morphism $X \to X^+$ contracting the curve $\Gamma=\P^1$ to a Wahl singularity. The minimal resolution of this Wahl singularity is obtained resolving minimally $X$ for singularities on $\Gamma$ and contracting the proper transform of $\Gamma$ and subsequent $(-1)$-curves (see \cite{HTU13} or \cite[\S1]{Urz13} for the numerics). The birational morphism $X_t \to X_t^+$ for $t\neq 0$ is just the blow down of one $(-1)$-curve. The inverse of a W-blow-down will be called a {\em W-blow-up}. We notice that over the Wahl singularity of $X^+$ we can have infinitely many W-blow-ups $X$. If the Wahl singularity in $X^+$ is $\frac{1}{\delta^2}(1,\delta a-1)$, then the new Wahl singularities (one or two) appearing in $X^+$ are of the type $\frac{1}{n_i^2}(1,n_i a_i-1)$ where $n_i=\delta n_{i-1}-n_{i-2}$ and $a_i=\delta a_{i-1}-a_{i-2}$ where $(n_0,a_0)=(0,1)$ and $(n_1,a_1)=(\delta,a)$ (see \cite{HTU13} or \cite[\S1]{Urz13}). When two singularities appear, then the numbers for the Wahl singularities are $(n_i,a_i)$ and $(n_{i+1},a_{i+1})$.

\textbf{(W-flip)} In the flipping type, the birational transformation $X \dashrightarrow X^+$ is described in \cite{HTU13} (see also \cite[\S1]{Urz13}). Roughly it is a minimal resolution of $X$ of the singularities contained in $\Gamma$, followed by blow downs and ups in the total transform of $\Gamma$, finishing with a contraction of two, one or none chain of rational curves. The numerical rules are in \cite{HTU13} (see also \cite[\S1]{Urz13}). Of course, the birational map $X_t \to X_t^+$ is an isomorphism for $t \neq 0$.

\begin{remark}
There are constraints for anti-W-flips and W-blow-ups. Anti-W-flips for $(X^+ \subset \X^+) \to (0 \in \D)$, which locally above $0$ is the deformation of an extremal $P$-resolution (see \cite[\S4]{HTU13}), only exist for certain deformations. (For the precise statement see \cite[Cor.3.23]{HTU13}) See Example 3.26 in \cite{HTU13} where the anti-flip is not an anti-W-flip. One can produce a similar example in the case of divisorial contractions, using a $\Z_6$-quotient of a simple elliptic singularity, whose minimal resolution is a union of $4$ smooth rational curves $E_1$, $E_2$, $E_3$, and $F$ so that the $E_i$ are disjoint, each meets the curve $F$ transversally at a single point, and $E_1^2=-2$, $E_2^2=-3$, $E_3^2=-6$, and $F^2=-2$. Similarly to \cite[Example 3.26]{HTU13}, here one attaches a $(-1)$-curve transversal at one point of $E_1$ to define a canonical (nonterminal) extremal neighborhood. The Wahl singularity in $X^+$ is $\frac{1}{9}(1,2)$.
\label{anti}
\end{remark}

When we apply \textbf{(III)} above, we have that either the Picard number of $X$, or the indices of some of its singularities strictly decreases. Thus it ends after finitely many steps with a W-surface as in \textbf{(I)} or \textbf{(II)}, or with $K_X$ nef.

\begin{definition}
Given a W-surface, a {\em minimal model} of it is a minimal W-surface obtained by applying W-blow-downs and W-flips.
\label{min}
\end{definition}

\begin{proposition}
A minimal model is unique.
\label{minmod}
\end{proposition}

\begin{proof}
Assume there are two minimal models $X_{m1}$ and $X_{m2}$. We have a birational map $g \colon \X_{m1} \dashrightarrow \X_{m2}$. Then, since the $\X_{mi}$ are $\Q$-factorial normal analytic $3$-folds with terminal singularities and $K_{\X_{mi}}$ nef, we know that $g$ can be written as composition of analytic flops \cite[Thm.4.9]{flops}. Let $g' \colon \X_{m1} \dashrightarrow \X'$ be the first flop. So, we have a small contraction $c \colon \X_{m1} \to \Y$ over $\D$ such that $K_{X_{m1}} \cdot C =0$ for the curves $C$ contracted by $c$ to $y \in \Y$. By \cite[Thm.2.2.2]{Ko91}, the singularity $y \in \Y$ is terminal, and by \cite[Thm.5.3]{KSB88} we obtain that $y \in Y$ is a T-singularity, where $Y$ is the special fiber of $\Y \to \D$. Locally $X_{m1} \to Y$ gives an M-resolution of $y \in Y$ and they are unique \cite{BC94}. Notice that for the contraction $\X' \to \Y$ we have the same situation, with same singularities. Therefore $g'$ is an isomorphism, and so $g$.
\end{proof}

In the previous lemma, $\Q$-factoriality for $\X$ was crucial. For example, take a rational elliptic surface with a multiplicity two fiber. Assume it has an $I_2$ fiber (nodal rational curve). We blow up its two nodes and contract the two $(-4)$-curves. One can show that there are $\Q$-Gorenstein smoothings and for any we have a minimal W-surface (a general fiber is an Enriques surface, see \cite[Sect.5]{Urz13}). Each of the two images of the $(-1)$-curves from the nodes of $I_2$ gives two different ways to contract, but these two contractions are not related by a flop. The $\Q$-Gorenstein smoothing of the T-singularity $[3,3]$ gives a $3$-fold which is not $\Q$-factorial.

If the minimal model has $K_X^2>0$, then we also have (unique) canonical model. (This is \cite[Lemma3.1]{Urz13}.)

\begin{proposition}
Let $X$ be a minimal W-surface with $K_X^2>0$. Then, its canonical model $(X_{\text{can}} \subset \X_{\text{can}}) \to (0 \in \D)$ has $X_{\text{can}}$ projective surface with only T-singularities. This is, it has Du Val singularities or cyclic quotient singularities $\frac{1}{dn^2}(1,dna-1)$ with gcd$(n,a)=1$. \label{canmod}
\end{proposition}

\begin{proof}
We know there is $(X_{\text{can}} \subset \X_{\text{can}}) \to (0 \in
\D)$; cf. \cite{KM1998}. We have a birational morphism $\X \to \X_{\text{can}}$ over $\D$ such that $K_{\X_{\text{can}}}$ is $\Q$-Cartier and ample. Notice that $X_{\text{can}}$ has log terminal singularities because $X$ has \cite[pp.102--103]{KM1998}. The singularities of $X_{\text{can}}$ must be T-singularities by \cite[\S5.2]{KSB88}.
\end{proof}


\section{Birational $\Q$-Gorenstein smoothings} \label{s2}

Let $X$ be a W-surface. Why do these smooth projective surfaces $X_t$ degenerate to $X$? We remark that in this type of degeneration there are no vanishing cycles. One can find an explanation when the general fiber has Kodaira dimension $0$ and $1$ in \cite[\S4]{Kaw92}. When $X$ is indeed singular, this has to do with collisions of special fibers on an elliptic fibration of $X_t$. Also, there is a general discussion for $q=p_g=0$ surfaces in \cite{Hack11e} through certain exceptional vector bundles. Here we describe the situation when a resolution of $X$ and $X_t$ have the same plurigenera. In particular, when they are actually birational. This happens for example with the surfaces in \cite{Man91,HP10}, and also in many examples of stable surfaces in the KSBA compactification of the moduli space of surfaces of general type; see \cite{Urz13}. We will work with the later at the end of this article.

The claim is: any such W-surface comes from a smooth W-surface, after applying W-blow-downs, and W-flips. Thanks to the MMP described in \S \ref{s1} and a result of Kawamata \cite[Lemma2.4]{Kaw92}, the only nontrivial case is degenerations of $\P^2$. We work out that now.

Let $\F_m$ be the Hirzebruch surface with a smooth rational curve $\Gamma_m$ with $\Gamma_m^2=-m$. Let $(0 \in \D)$ be a smooth analytic germ of curve. We consider three situations $m=3,5,7$:

\begin{itemize}
\item[\textbf{(dm)}] A W-surface with $X=\F_m$, $X_t=\F_1$ for $t \neq 0$, and the $(-1)$-curve of $\F_1$ degenerates to $\Gamma_m + \sum_{i=1}^{(m-1)/2} F_i$, where $F_i$ are distinct fibers of $\F_m \to \P^1$.
\end{itemize}

\begin{theorem}
Let $X$ be a W-surface with general fiber $\P^2$. Then, after applying one W-blow-up and finitely many W-flips, this W-surface becomes the W-surface in \textbf{(dm)} for some $m$.
\label{projplane}
\end{theorem}

The sequence of W-flips is a guided process as we will see in the proof. For the inverse construction, we remark that anti-W-flips are not automatic from the corresponding ``local" extremal P-resolution, they only exist for certain deformations in the deformation space of this P-resolution (see \cite{HTU13}). Moreover, for a given a P-resolution, there are in general infinitely many combinatorial choices for an anti-flip.

First we recall Manetti's \cite[Thm.18]{Man91}. Let $X$ be a W-surface with general fiber $\P^2$, let $\delta \colon S \to X$ be its minimal resolution. Then there is a rational fibration $p \colon S \to \P^1$, such that the exceptional set of $\delta$ is contained in either $0$, or $1$, or $2$ singular fibers union one section $\Gamma$. Moreover, there is a birational morphism $\mu \colon S \to \F_m$ contracting the curves in the before mentioned special $1$ or $2$ fibers such that $\Gamma$ does not intersect the exceptional loci of $\mu$, the image of $\Gamma$ is the negative curve of $\F_m$, and $m$ is either $4$, $7$, or $10$ in correspondence to the number of special fibers $0$, $1$, $2$.

\textbf{First W-blow-up:} Let us consider a section of $\X \to \D$ intersecting $X$ at a point which is contained in a nonsingular fiber of $p \colon S \to \P^1$. Blow this section up, let $X_0$ be this new W-surface whose general fiber is now $\F_1$. If $S_0$ is the blow up at the point of $S$, then we have created a $(-1)$-curve intersecting $\Gamma$ at one point. This is a flipping \eni ~for the Wahl singularity which includes $\Gamma$ in its minimal resolution. We apply the corresponding W-flip. Notice that $\Gamma^2 \leq -4$, and so for the new W-surface $X_1$ we have a minimal resolution $\delta_1 \colon S_1 \to X_1$ with the same properties as for $\delta \colon S \to X$ but now $m$ is either $3$, $6$, or $9$ (in correspondence to the number of special fibers $0$, $1$, $2$).

The birational operations we are going to use are W-flips with flipping curves always in the special (singular) fibers of $p \colon S \to \P^1$. For this we need the following.

A $\maltese$ type of surface will be a W-surface $X_i$ with four distinguished rational smooth curves $A$, $B$, $C$, and $F$ so that the minimal resolution $\delta_i \colon S_i \to X_i$ has a rational fibration $p_i \colon S_i \to \P^1$ with at most two singular fibers (as the one above), where $A$ is a $(-1)$-curve in one singular fiber, $B$ and $C$ are in the other, and $F$ is a general smooth fiber. In addition they must satisfy:

\begin{enumerate}
\item $B \cdot K_{X_i} <0$ and $C \cdot K_{X_i} >0$.
\item We can contract $B$ in $X$ and $C$ in $X$ to cyclic quotient singularities.
\item The special section $\Gamma$ of $S_i$ appears in one of the minimal resolutions of the cyclic quotient singularities in (2) with $\Gamma^2 \leq -3$.
\end{enumerate}

The conditions (1) and (2) ensure the existence of an \eni ~or \enii ~of flipping type for the curve $B$. Notice that, in our case, it is of flipping type since the $(-1)$-curve of the general fiber $\F_1$ degenerates to an effective $\Q$-divisor in $X$ containing $F$. The condition (3) says that we preserve the fibration of $S_i$, and so the strict transform of $F$ after the flip is again an $F$ for $S_{i+1}$. In particular we are flipping curves only in the two special fibers.

Before continuing with the proof, we run the MMP of Theorem \ref{projplane} in an example.

\begin{figure}[htbp]
\includegraphics[width=8cm]{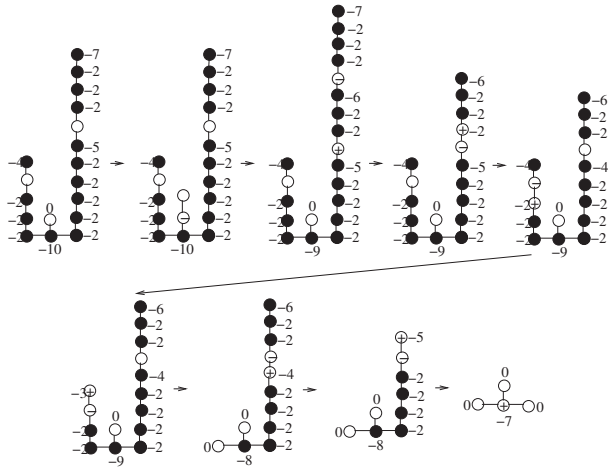}
\caption{Running Theorem \ref{projplane} on an example} \label{f0}
\end{figure}

\begin{example}
We take the example in \cite[p.117]{Man91}. There $X$ has three Wahl singularities $\frac{1}{4}(1,1)$, $\frac{1}{25}(1,4)$, and $\frac{1}{29^2}(1,29 \cdot 21 -1)$. The minimal resolution $\delta \colon S \to X$ gives the fibration $p \colon S \to \P^1$ with the exceptional divisor of $\delta$ contained in two fibers and one section $\Gamma$, where $\Gamma^2=-10$. In the upper-half corner of Figure \ref{f0}, we have the dual graph of the exceptional curves of $\delta$ together with two $(-1)$-curves (one in each of the two fibers), and a fiber. The $\bullet$ represent the curves in the exceptional divisor, and the numbers are the self-intersections of the curves (no number means $-1$). As in \cite{HTU13,Urz13}, the graph-diagram in Figure \ref{f0} shows how we run the MMP in Theorem \ref{projplane}. First, we have the W-blow-up of a section of the W-surface $X$ (i.e. of the $3$-fold) which intersects a point of the distinguished fiber not in $\Gamma$. The first arrow indicates that birational transformation on the surface $S$. After that, we perform seven W-flips. The $K$-negative curves are marked with a $\ominus$, the $K$-positive flipped ones with a $\oplus$. After the first flip, we see the curves $A$, $B$, $C$, $F$ above as some permutation of the elements in $\{ \circ, \circ, \oplus, \ominus \}$.
\label{e1}
\end{example}

\begin{lemma}
Assume that $X_i$ is a W-surface of $\maltese$ type coming from W-flips starting with Manetti's situation above. Then, the W-flip of $B$ in the W-surface $X_i$ produces a W-surface $X_{i+1}$ of $\maltese$ type.
\label{keyplanelemma}
\end{lemma}

\begin{proof}
In $X_{i+1}$, let $A$ be the strict transform of $A$, same for $C$, and let $B$ be the flipped $B$. So we know that $B \cdot K_{X_{i+1}} >0$. We notice that (3) is clear from the Manetti's assumptions on $S_i \to \P^1$ and the way a W-flip is performed on $X_i$ (it is the contraction of at most two Wahl chains of a blow-up on consecutive nodes from the minimal resolution of the corresponding cyclic quotient singularity). Now $B$ could be in the same singular fiber (in $S_{i+1}$) than either $A$ or $C$. We rename curves so that the flipped curve $B$ is $C$, the curve in the same singular fiber than the flipped curve is $B$, and the curve in the other singular fiber is $A$. In this way, $A$ has to be a $(-1)$-curve in the minimal resolution $S_{i+1}$.

To prove (1), we need to show $B \cdot K_{X_{i+1}} < 0$. For this we use that $X_{i+1}$ has Picard number $2$. Notice that because of (3), $F \cdot K_{X_{i+1}} <0$, and so $F \equiv uB + vC$ with $u,v \in \Q$. We have two cases: $A \cdot B=0$ or $A \cdot C=0$. Say $A \cdot C=0$, then $A \cdot B>0$ and so $F \cdot A = u B \cdot A$ implies $u>0$. Now $u B^2 + v C \cdot B  = F \cdot B$ implies so $v>0$. Therefore $B \cdot K_{X_{i+1}}<0$. For the other case $A \cdot B=0$, we do similar intersections to obtain the same result, so (1) is true.
\end{proof}

\begin{proof}[Proof of Theorem \ref{projplane}]
Notice that the $X_1$ above (obtained from \cite[Thm.18]{Man91} applying one W-blow-up and one W-flip) is of $\maltese$ type. Then we apply the W-flips for the W-surfaces $X_i$ of $\maltese$ type applying previous lemma. They have to terminate on a nonsingular $F_n$ W-surface for some $n$. Notice that $n=3,5,7$ according to the number of special fibers. Also, the $(-1)$-curve of the general fiber $\F_1$ degenerates to the claimed curves in $\F_n$.
\end{proof}
\bigskip

Let us denote by $P_m(Z):=\text{dim}_{\C} H^0(Z,m K_Z)$ the plurigenera of a nonsingular projective surface $Z$.

\begin{theorem}
Let $X$ be a W-surface, and let $\widetilde{X}$ be a resolution of the singularities of $X$. Assume that $P_m(\widetilde{X})=P_m(X_t)$ for $t \neq 0$ and $m \gg 0$. Then this W-surface can be constructed from a smooth W-surface applying a finite number of W-blow-downs (or -ups) and anti-W-flips.
\label{birWsurf}
\end{theorem}

\begin{proof}
We apply the MMP of \S \ref{s1} to this W-surface. If we arrive to (II), then we are done. If we arrive to (I), then we are done by Theorem \ref{projplane}. Otherwise, after finitely many W-flips or W-blow-downs, we obtain the W-minimal model $X'$ of $X$. If the W-minimal model is smooth, then we are done. If not, by \cite[Lemma2.4]{Kaw92}, there are positive integers $m_1$ and $m_2$ such that $$P_m(X_t) > P_m(\widetilde{X'})$$ for $t \neq 0$ holds for positive integers $m$ with $m_1$ dividing $m$ and $m_2 <m$, where $\widetilde{X'}$ is a resolution of $X'$. But $P_m(\widetilde{X})=P_m(\widetilde{X'})$ and $P_m(X_t)=P_m(X'_t)$ for $t \neq 0$ and all $m$, because plurigenera is a birational invariant between nonsingular varieties. By assumption $P_m(\widetilde{X})=P_m(X_t)$ for $t \neq 0$ and $m\gg 0$, and so we have a contradiction. Therefore $X'$ must be nonsingular.
\end{proof}

\begin{corollary}
Let $X$ be a W-surface such that $X$ is birational to $X_t$, $t \neq 0$. Then this W-surface can be constructed from a smooth W-surface applying a finite number of W-blow-downs (or -ups) and anti-W-flips.
\label{birWsurfA}
\end{corollary}

\qed

The difficulty of the $\Q$-Gorenstein smoothings in Theorem \ref{birWsurf} relies mainly on W-flips and W-blow-downs. As we have noticed before, there are choices when one performs inverses of them. Also, their description uses an intricate but explicit combinatoric through continued fractions (cf. \cite{HTU13}). On the other hand, the degenerations in Theorem \ref{birWsurf} produce interesting degenerations of curves in the general fiber $X_t$. This is the topic of the next two sections.

\section{Degeneration of curves} \label{s3}

Here is the situation. Let $X$ be a W-surface. Let $\Gamma_t$ be an irreducible curve in the general fiber $X_t$. This curve deforms to a divisor $\Gamma_0$ in $X$. If $X'$ a W-minimal model of $X$, then we want to know explicitly about the degeneration of $\Gamma'_t$ into $\Gamma'_0$ (the total transforms of $\Gamma_t$ and $\Gamma_0$ respectively). Moreover, if $X$ and $X_t$ are birational, we can get a nonsingular $X'$ (Corollary \ref{birWsurfA}) and typically we would have $X'$ isomorphic to $X'_t$. This suggests we can think about this deformation as an explicit degeneration of $\Gamma'_t$ in $X'$. At the end of this section we will show examples using curves in $\P^2$. In the next section we will see how these sort of examples show up in the Koll\'ar--Shepherd-Barron--Alexeev boundary of the moduli space of surfaces of general type.

The set-up: Let $X$ be a W-surface, and let $\Gamma_t$ be an irreducible curve in $X_t$, for $t \neq 0$, degenerating to a divisor $\Gamma_0$ in $X$. This is, we have a divisor $\Gamma$ in $\X$ such that $\Gamma|_{X_t} = \Gamma_t$. We write $\Gamma_0 = \sum_{i=0}^s \alpha_i D_i$ where $D_i$ are distinct irreducible curves and $\alpha_i \in \Z_{\geq 0}$.

Let $(E^- \subset X \subset \X) \to (Q \in Y \subset \Y)$ be a \eni ~or \enii. Assume $D_0=E^-$. If it is of divisorial type, then the W-blow-down produces a W-surface $X'$ with a divisor $\Gamma'$ image of $\Gamma$, such that $\Gamma'_t$ is an irreducible curve for $t \neq 0$ (or a point if $\Gamma$ is the exceptional divisor), and $\Gamma'_0= \sum_{i=1}^s \alpha_i D'_i$ where $D'_i$ is the image of $D_i$.

If it is of flipping type, then the W-flip produces a W-surface $X'$ with a divisor $\Gamma'$ proper transform of $\Gamma$, such that $\Gamma'_t$ is isomorphic to $\Gamma_t$ for $t \neq 0$, and $\Gamma'_0= \sum_{i=1}^s \alpha_i D'_i + \beta E^+$ where $D'_i$ is the proper transform of $D_i$, $E^+$ is the flipping curve, and $\beta \in \Z_{\geq 0}$. Since $\X'$ is $\Q$-factorial and this is a flip, we have the numerical equivalence $$\sum_{i=1}^s \alpha_i (D'_i \cdot K_{X'}) + \beta (E^+ \cdot K_{X'}) = \Gamma'_0 \cdot K_{X'} = \Gamma'_t \cdot K_{X'_t} = \Gamma_t \cdot K_{X_t} $$ from where one calculates $\beta$. We notice that $\beta=0$ if and only if $\Gamma_0 \cap E^- = \emptyset$. The following is a frequent case.

\begin{proposition}
Assume that in the situation above we have an \eni ~of flipping type so that $E^-$ is intersecting at one point one of the ends of the minimal resolution of the corresponding Wahl singularity. Let $\Gamma_0=\sum_{i=1}^s \alpha_i D_i$ with $D_i \neq E^-$ for all $i$, $D_1 \cdot E^- =1$, and $D_i \cdot E^- =0$ for $i \geq 2$ (see Figure \ref{f1}). Then $\beta=\alpha_1$.
\label{freqFlip}
\end{proposition}

\begin{figure}[htbp]
\includegraphics[width=8cm]{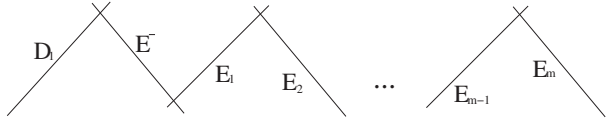}
\caption{Frequent case}  \label{f1}
\end{figure}

\begin{proof}
Let $\frac{1}{n^2}(1,na-1)$ be the Wahl singularity, and let $E_1, E_2, \ldots, E_m$ be the chain of exceptional curves in the minimal resolution as shown in Figure \ref{f1}. Thus $E_1 \cdot E^-=1$, where as always $E^-$ is the strict transform of $E^-$, so it is a $(-1)$-curve. Assume that $E_1, \ldots, E_r$ are all $(-2)$-curves (this set may be empty). Then either $E_m^2=-(r+2)$ if $r \neq m-1$, or $E_m^2=-(r+4)$ if $r =m-1$. The W-flip in this case gives $E^+=E_m$ with either one Wahl singularity on it if $r \neq m-1$ having the Hirzebruch-Jung continued fraction $[-E_{r+1}^2-1,-E_{r+2}^2,\ldots,-E_{m-1}^2]$, or none \cite[Prop.2.8]{Urz13}. Let $a$ and $b$ the discrepancies ($0<a,b<1$) of $E_{r+1}$ and $E_{m-1}$ for the singularity of $X'$ in $E^+$. Then $K_{X'} \cdot D'_1= K_X \cdot D_1 -(r+1)+a$ and $K_{X'} \cdot E^+= r + b$ if $X'$ singular, and so in this case $$\beta =  \frac{\alpha_1(r+1-a)}{r+b}=\alpha_1$$ because $a+b=1$. When $X'$ is smooth at $E^+$, $K_{X'} \cdot D'_1= K_X \cdot D_1 -(r+1)$ and $K_{X'} \cdot E^+= r + 1$, so $\beta =\alpha_1$.
\end{proof}

We will finish this section with two examples derived from vanishing of certain cohomology group.

\begin{proposition}
Let $Z$ be a nonsingular projective surface. Let $\W_1$,..., $\W_r$ be $r$ chains of $\P^1$'s in $Z$ so that $\W_i$ is the exceptional divisor of a Wahl singularity, and $\W_i \cap \W_j= \emptyset$ for $i\neq j$. Assume there is a curve $\Gamma_0=\P^1$ in $Z$ such that $\Gamma_0 \cap \W_i = \emptyset$ for all $i$, and $$H^2\big(Z,T_Z(-\log(\W_1+\ldots+\W_r+\Gamma_0))\big)=0.$$ Then there is a W-surface $X$ such that $X$ is the contraction of all $\W_i$'s, and there is a divisor $\Gamma \subset \X$ with $\Gamma|_{X_t}=\P^1$ and $\Gamma|_X=\Gamma_0$.

\label{obstr1}
\end{proposition}

\begin{proof}
We first blow-up distinct points on $\Gamma_0$ so that $\Gamma_0^2 \leq -2$, if necessary. By the adding-deleting $(-1)$-curves procedure (for instance, see \cite[\S4]{PSU13}), we have again $H^2(Z',T_{Z'}(-\log(\W_1+\ldots+\W_r+\Gamma_0)))=0$ where $Z'$ is the blown-up surface, and $\Gamma_0$ is the proper transform of $\Gamma_0$. The configurations $\W_1$,..., $\W_r$, $\Gamma_0$ correspond to exceptional divisors of cyclic quotient singularities, and they can be contracted to a projective surface $X'$. By \cite[pp.487--488]{LP07}, we have that $H^2(Z',T_{Z'}(-\log(\W_1+\ldots+\W_r+\Gamma_0)))=0$ implies $H^2(X',T_{X'})=0$, and that implies no-local-to-global obstructions to deform $X'$. In particular, we can glue local $\Q$-Gorenstein smoothings for the Wahl singularities $\W_i$, and keep the singularity corresponding to $\Gamma_0$. Let $\X' \to \D$ be the corresponding deformation. Since the deformation for the singularity given by $\Gamma_0$ is trivial, we can resolve it simultaneously, obtaining $\X'' \to \D$ $\Q$-Gorenstein smoothing of $X''$ (which is $X'$ with the singularity associated to $\Gamma_0$ minimally resolved). The $3$-fold $\X''$ has a divisor $\Gamma''$ defined by the curve $\Gamma_0$. The $(-1)$-curves in $X'$ corresponding to the initial blow-up of $W$ intersect $\Gamma_0$ at one point. All of them are now contracted to finally obtain the W-surface $X$ with the claimed properties.
\end{proof}

\begin{proposition}
Let $\sum_{i=1}^n C_i$ be a SNC divisor (only nodes as singularities, $C_i$ nonsingular curves) in a nonsingular projective surface $W$ with $H^0(W,\Omega_W^1)=0$. Assume that there is $1<m\leq n$ such that $C_1+\ldots+C_m \sim -K_W$, and the curves $\{C_{m+1}, \ldots, C_n\}$ are numerically independent. ($m=n$ means no second requirement.) Then $$H^2\big(W,T_W(-\log(C_1+\ldots+C_n))\big)=0.$$

\label{obstr2}
\end{proposition}

\begin{proof}
By Serre's duality, we need to prove that $$H^0(W,\Omega_W^1(\log(C_1+\ldots+C_n))\otimes \O_W(K_W))=0.$$ Notice that we have the natural {\small $$\Omega_W^1 \big(\log( \sum_{i=j}^n C_i) \big) \otimes \O_W(\sum_{i=1}^{j-1}C_i + K_W) \hookrightarrow \Omega_W^1 \big(\log( \sum_{i=j+1}^n C_i) \big) \otimes \O_W(\sum_{i=1}^{j}C_i + K_W),$$}and so by the hypothesis $C_1+\ldots+C_m \sim -K_W$, we want to show $H^0(W,\Omega_W^1(\log( C_{m+1}+ \ldots +C_n)))=0$. The long exact sequence in cohomology of the residue short exact sequence $$0 \to \Omega_W^1 \to \Omega_W^1(\log( C_{m+1}+ \ldots +C_n)) \to \bigoplus_{i=m+1}^n \O_{C_i} \to 0 $$ gives the Chern class map $\bigoplus_{i=m+1}^n H^0(C_i,\O_{C_i}) \to H^1(W,\Omega_W^1)$ (cf. \cite[pp.454--462]{GH}), and since we are assuming that $\{C_{m+1}, \ldots, C_n\}$ are numerically independent, this map is injective. The extra assumption $H^0(W,\Omega_W^1)=0$ finishes the proof.
\end{proof}

We present two examples to illustrate.

\begin{example} This is a way to produce a degeneration of a nodal rational plane curve of degree $d$ to $d$ lines in general position. Let $\{L_1, \ldots, L_d \}$ be $d$ lines in general position in $\P^2$. We assume $d \geq 6$ (this allows as a short uniform argument, this can be adapted for $d<6$). We blow-up $\sigma \colon W \to \P^2$ all the ${d \choose 2}$ intersection points of this line arrangement. Let $E_1$, $E_2$, $E_3$ be the exceptional curves over the three nodes  of the triangle $L_1$, $L_2$, $L_3$. Then $-K_W \sim E_1+E_2+E_3+L_1+L_2+L_3$ and $\{L_4, \ldots, L_d\}$ are numerically independent. Thus by Prop. \ref{obstr2}, $H^2(W,T_W(-\log(E_1+E_2+E_3+L_1+\ldots+L_d)))=0$. We can delete the $(-1)$-curves $E_1$, $E_2$, $E_3$ to have $H^2(W,T_W(-\log(L_1+\ldots+L_d)))=0$. Notice that $L_i^2=2-d \leq -4$. Let $F_i$ be the exceptional curve between $L_1$ and $L_i$, $i=2,\ldots,d$. We blow-up $d-6$ times over $L_1 \cap F_i$ to obtain a Wahl configuration $\W_i$ $[2,\ldots,2,d-2]$ for each $i$. Adding-deleting $(-1)$-curves (see \cite[\S4]{PSU13}) keeps obstruction zero. If $L_1=:\Gamma_0$, then we are in the situation of Prop. \ref{obstr1}, and so we have such a W-surface $X$. The surface $X$ has $d-1$ Wahl singularities of the same type. The curve $\Gamma_0$ is $D_i$ in Prop. \ref{freqFlip}. We perform $d-1$ W-flips using the $d-1$ $(-1)$-curves over the $L_i \cap F_i$'s, obtaining a smooth W-surface $X'$ with the $\Gamma_t=\P^1$ degenerating to $\sum_{i=1}^d L_i$. To obtain the planar degeneration, we blow-down the $(-1)$-curves corresponding to the $L_i \cap L_j$ with $i,j$ different than $1$. There are ${d \choose 2}-(d-1)=\frac{(d-1)(d-2)}{2}$ of them. Each of them intersect $\sum_{i=1}^d L_i$ at two distinct points, and so the corresponding $(-1)$-curve and $\Gamma_t$ intersect t two distinct points, this is, each creates a node for $\Gamma_t$. After we blow down all of them, we arrive to $\P^2$, and we have the degeneration we wanted.
\label{lines}
\end{example}

\begin{example}
Fix an integer $n\geq 1$. In this example we degenerate a rational curve of degree $3(n+1)$ with $9$ singular points of multiplicity $n+1$ and one node to two general nodal cubics, one with multiplicity $1$, and the other with $n$. We start with two such nodal cubics $C_1,C_2$. We blow up the $9$ base points and the two nodes; let $Z$ be the surface. Let $E_i$ be the $(-1)$-curve from the node of $C_i$. Using similar arguments as in Prop. \ref{obstr2}, we can show that $H^2(Z,T_Z(-\log(C_1+C_2)))=0$. We now blow up $n-1$ times over one point of $C_2 \cap E_2$ to obtain the Wahl configuration $[2,\ldots,2,n+4]$ in a surface $Z'$. Let $F_1=E_2, \ldots, F_{n-1}$ be the $(-2)$-curves in a chain, so $F_1 \cap C_2 \neq \emptyset$. Then $H^2(Z',T_{Z'}(-\log(C_1+C_2+F_1+ \ldots + F_{n-1})))=0$. We contract $C_2+F_1+ \ldots + F_{n-1}$ to obtain $X'$, and use Prop. \ref{obstr1} with $\Gamma_0=C_1$ to produce a W-surface $X$ with the divisor $\Gamma$. We now W-flip one of the $9$ exceptional curves of the $9$ base points between $C_1$ and $C_2$, and then we W-flip $F_n$, $F_{n-1}$, ..., $F_2$ using the simple flips of Prop. \ref{freqFlip} (see Figure \ref{f2}).

\begin{figure}[htbp]
\includegraphics[width=10cm]{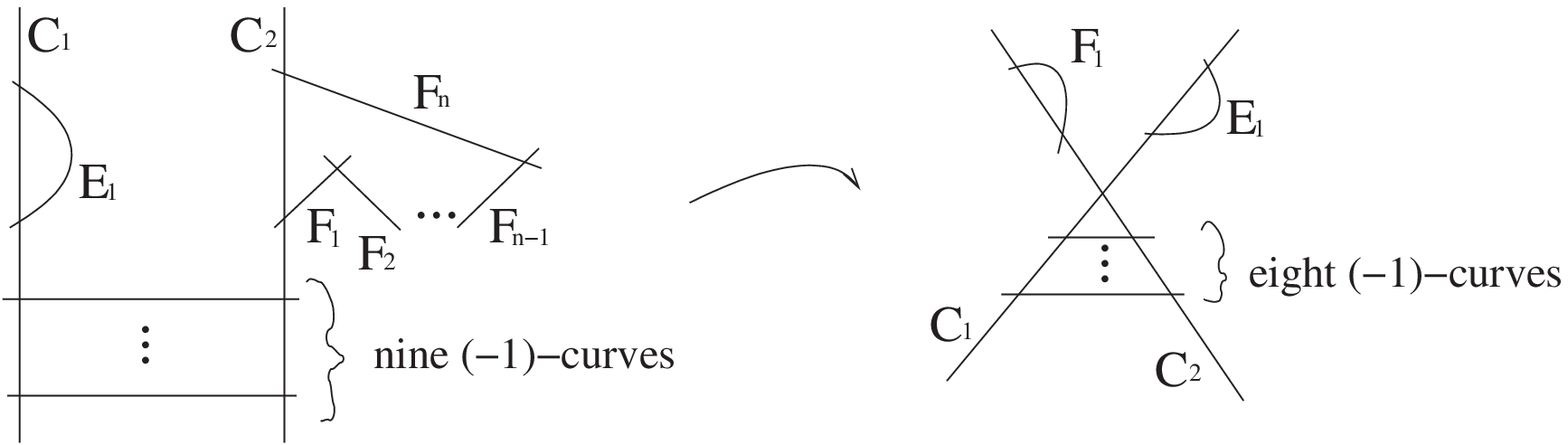}
\caption{}  \label{f2}
\end{figure}

One can compute that $\Gamma_t$ degenerates to $C_1+nC_2+(n-1)F_1$. This corresponds to a nonreduced degeneration of $\Gamma_t$. For a model in $\P^2$ of this deformation, we blow-down $E_1$, $F_1$, and the $8$ remaining $(-1)$-curves from the $9$ base points. Using intersection numbers, we see that $E_1$ in $X_t$ is a $(-1)$-curve intersecting transversally at two points of $\Gamma_t$, and for the other $9$ $(-1)$-curves (including $F_1$), we obtain exceptional $(-1)$-curves intersecting $\Gamma_t$ with multiplicity $n+1$. Hence in $\P^2$ the curve $\Gamma_t$ has degree $3(n+1)$ and the claimed multiplicities for its singularities.
\label{halphen}
\end{example}

\section{Explicit rational examples of general type} \label{s4}

We are going to give four examples where birational $\Q$-Gorenstein deformations appear when studying the boundary of the KSBA compactification of the moduli space $\overline{M}$ of simply connected surfaces of general type with $p_g=0$ and $K^2=1$. We point out that the choice of the invariants is irrelevant for the techniques.

\subsection{Irreducible septic associated to a $\frac{1}{36}(1,5)$} In this example we show how to find explicitly a curve in $\P^2$ associated to a boundary divisor in $\overline{M}$ which parametrizes rational surfaces with one fixed Wahl singularity. For that we use a further degeneration on two or more Wahl singularities; this is the case in all examples analyzed in \cite{Urz13}.

\begin{figure}[htbp]
\includegraphics[width=8cm]{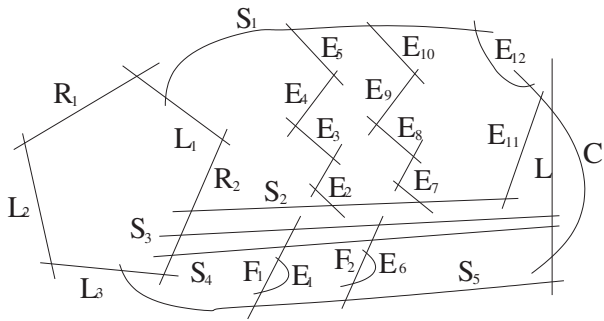}
\caption{Elliptic fibration with singular fibers $I_5,I_2,5I_1$}  \label{f4}
\end{figure}

Here we start with the example in \cite[Fig.6]{LP07}. We consider that example on a more general elliptic fibration. Let us take the plane configuration in Figure \ref{f4}: $L_1,L_2,L_3,L$ are lines with a triple point $L \cap L_1 \cap L_2$, and $C$ is a conic forming a triple point at $L_1 \cap L_3$. All other intersections are general. The pencil $\{a \cdot C\cdot L + b \cdot L_1 \cdot L_2 \cdot L_3 \colon [a,b] \in \P^1\}$ defines an elliptic fibration with singular fibers $I_5,I_2,5I_1$ as shown in Figure \ref{f4}. In that figure we drew two $I_1$'s and five sections which will be used later. Notice that they exists with the shown intersections. We blow-up twelve times this elliptic fibration to obtain a surface $Z$ and the configuration in Figure \ref{f5}. The corresponding exceptional divisors are $E_1,\ldots,E_{12}$, we blew-up in that order. The five Wahl configurations in \cite[Fig.6]{LP07} correspond to $S_1,L_1,R_1,L_2,L_3=[8,2,2,2,2]$ (numbers are the corresponding minus self-intersections), $F_1,E_3,E_4=[6,2,2]$, $F_2,E_8,E_9=[6,2,2]$, $S_2=[4]$, and $C=[4]$. Let $\sigma \colon Z \to X$ be the contraction of these five configurations to a normal projective surface $X$ (with five Wahl singularities). As in \cite{LP07}, this surface $X$ has no local-to-global obstructions to deform, and $X$ has nef canonical class. Hence W-surfaces $X$ exist, and have general fiber a simply connected surface of general type with $K^2=1$ and $p_g=0$.

Let $X'$ be the minimal resolution of the singularity $[8,2,2,2,2]$ of $X$. Consider a W-surface $X'$ keeping the configuration $[8,2,2,2,2]$. The contraction of this configuration in the general fiber gives a surface with one Wahl singularity in the corresponding boundary divisor of $\overline{M}$. We now perform W-flips on the W-surface $X'$ to find the minimal model of its general fiber. It turns out to be rational, and so we instead find a way to blow-down enough $(-1)$-curves to arrive to $\P^2$,  keeping track of the configuration $[8,2,2,2,2]$. Let $\Gamma_t$ be the $(-8)$-curve in this configuration.

\begin{figure}[htbp]
\includegraphics[width=9cm]{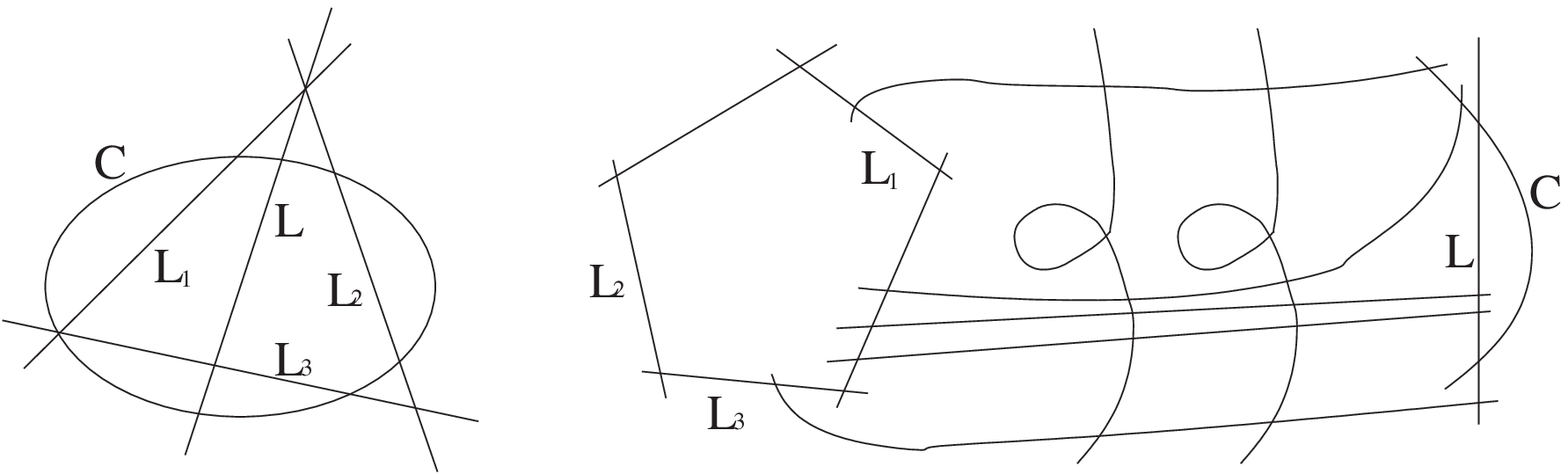}
\caption{}  \label{f5}
\end{figure}

We first flip the curve $E_5$ (in the threefold of the W-surface $X'$). This curve is as in Prop. \ref{freqFlip}, and so the flipped W-surface $X'$ (abuse of notation) has flipped curve $F_1$ and $X'$ has only three singularities. Now we W-flip the curve $E_{10}$ producing an analog situation: $F_2$ is the flipped curve and we have two Wahl singularities $\frac{1}{4}(1,1)$. We now W-flip $E_{12}$ and then $E_{11}$, the resulting W-surface $X'$ is now smooth. This shows that the general fiber is indeed rational. By a repeated use of Prop. \ref{freqFlip}, we compute that $\Gamma_t$ degenerates to $$\Gamma_0=S_1+F_1+F_2+C+S_2.$$ The surface $X'$ has $K_{X'}^2=1-5=-4$, and to get into $\P^2$ is enough to blow-down $13$ times $(-1)$-curves. In the W-surface $X'$ this corresponds to trivial divisorial contractions driving the general fiber to $\P^2$ as well. The $13$ curves we are going to blow-down in $X'$ are $S_5$, $L_3$, $L_2$, $R_1$, $L_1$, $E_2$, $E_7$, $S_2$, $C$, $S_3$, $S_4$, $E_1$, and $E_6$. Each of them induces a blow-down in the general fiber, whose exceptional divisors are shown in Figure \ref{f6}, we use same letters. This is found by intersecting these curves with $\Gamma_0$, then we know the intersection with $\Gamma_t$. The curves $S_2$ and $C$ may be as in Figure \ref{f6} or tangent to the $(-8)$-curve (double points anyway).

\begin{figure}[htbp]
\includegraphics[width=8cm]{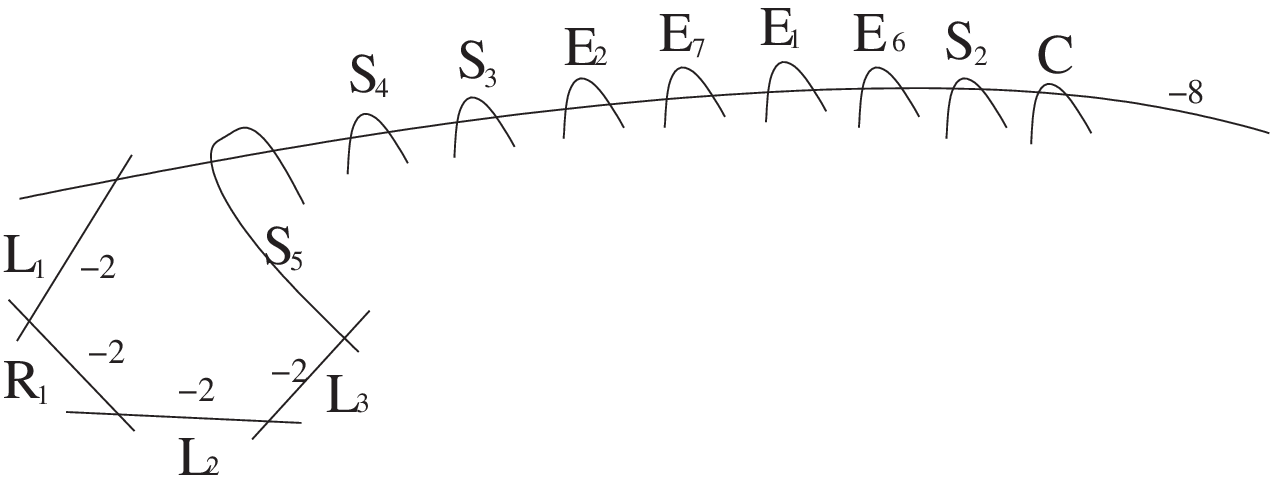}
\caption{}  \label{f6}
\end{figure}

We are now in good shape to blowing down the $13$ divisors. The image of the configuration of curves in $X'$ becomes two nodal cubics $F_1$ and $F_2$ together with a line $S_1$. The two cubics are tangent at two points $P$ and $Q$ with multiplicity $5$ and $2$, and the line passes through $P$ and $Q$ transversally to each cubic. All other intersections are nodes. For the image of $\Gamma_t$ we obtain an irreducible septic with $8$ double points, and one $D_{10}$ singular point (i.e. locally of the type $\{x(y^2+x^8)=0\} \subset \C^2$).

Therefore, the degeneration in the boundary of $\overline{M}$ gives, up to birational (discrete) flips and divisorial contractions, a degeneration (continuous) of the above irreducible septic into the above reducible (but reduced) septic. The irreducible septic represents (up to birational transformations) a boundary divisor (of dimension $7$) in $\overline{M}$ associated to the singularity $[8,2,2,2,2]$.

\subsection{Degree $15$ curve for a $\frac{1}{16}(1,3)$ and nonreduced degeneration} In the previous example, and the examples below, we will see septics showing up. We chose those examples to put things simple, but naturally higher degree curves appear degenerating to a reducible septics with multiplicities on certain components. For instance, take the above example and work out the singularity $[6,2,2]$ (for the $[4]$ we do not obtain rational surfaces but Dolgachev $(2,3)$ surfaces; see \cite[Prop.2.3]{Urz13}). Say $[6,2,2]$ is $F_1+E_3+E_4$. We do what we did above for $[8,2,2,2,2]$ but now for $[6,2,2]$. Then, after some flips we get a smooth W-surface, and if we blow-down to $\P^2$ in the same way as before, we get an irreducible (rational) curve of degree $15$. This curve, which is the image of the $(-6)$-curve of $[6,2,2]$, degenerates to $3S_1+F_1+3F_2$, this is, degenerates to the same septic configuration above (two nodal cubics and a line) but now with multiplicities.

\subsection{One septic degeneration for two distinct deformations} Here we produce two distinct degenerations of stable surfaces in the boundary of $\overline{M}$ which correspond to the same degeneration of plane curves. This is, the continuous part could be seen as the same (but the discrete part ($3$-fold birational transformations) is of course different).

\begin{figure}[htbp]
\includegraphics[width=10cm]{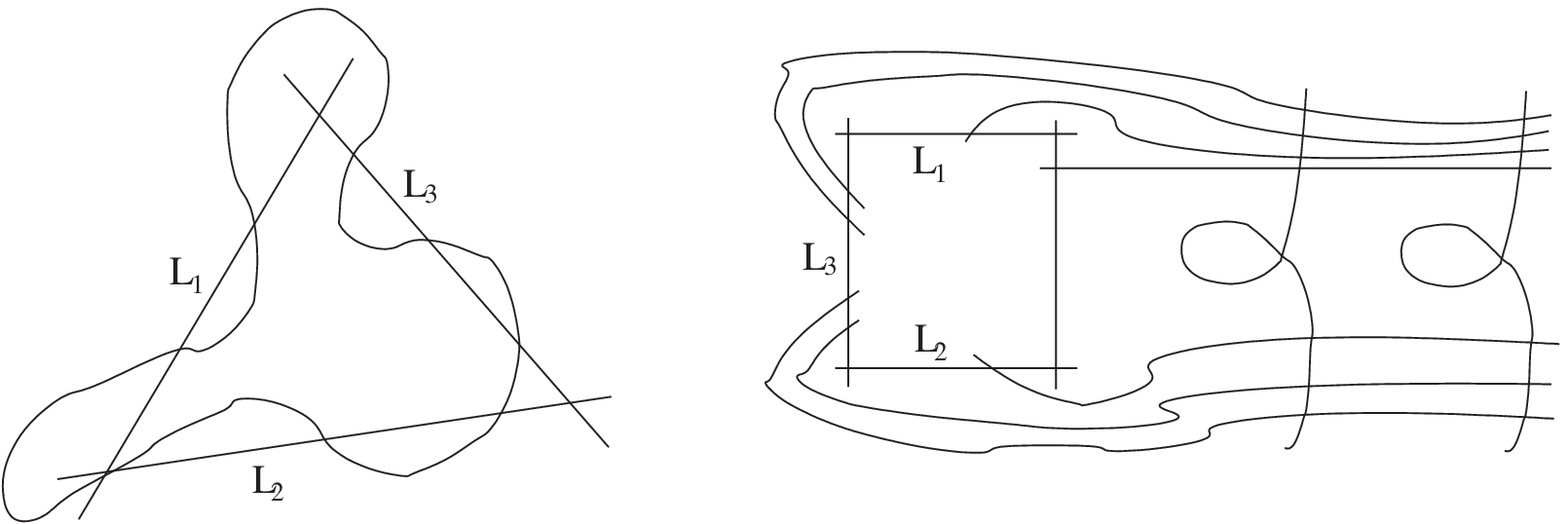}
\caption{}  \label{f7}
\end{figure}

We start with a triangle $L_1, L_2, L_3$ and a general cubic passing transversally through $L_1 \cap L_2$ as in Figure \ref{f7}. The triangle and the cubic defines a pencil which induces an elliptic fibration with sections whose singular fibers are $I_4$, and $8$ $I_1$'s. In Figure \ref{f7} we show the $I_4$, two chosen $I_1$'s and $7$ chosen sections. We blow-up on this configuration in two different ways, obtaining surfaces $Z_i$ with $i=1,2$. For $Z_1$ we blew-up $11$ times, for $Z_2$ $13$ times. The corresponding exceptional divisors are $E_i$ and $G_i$ respectively, the subindex $i$ indicates the order of the blow-ups. This is shown in Figure \ref{f8}. We name other irreducible curves with $S_i$, $F_i$, and $R$ as shown.

\begin{figure}[htbp]
\includegraphics[width=12cm]{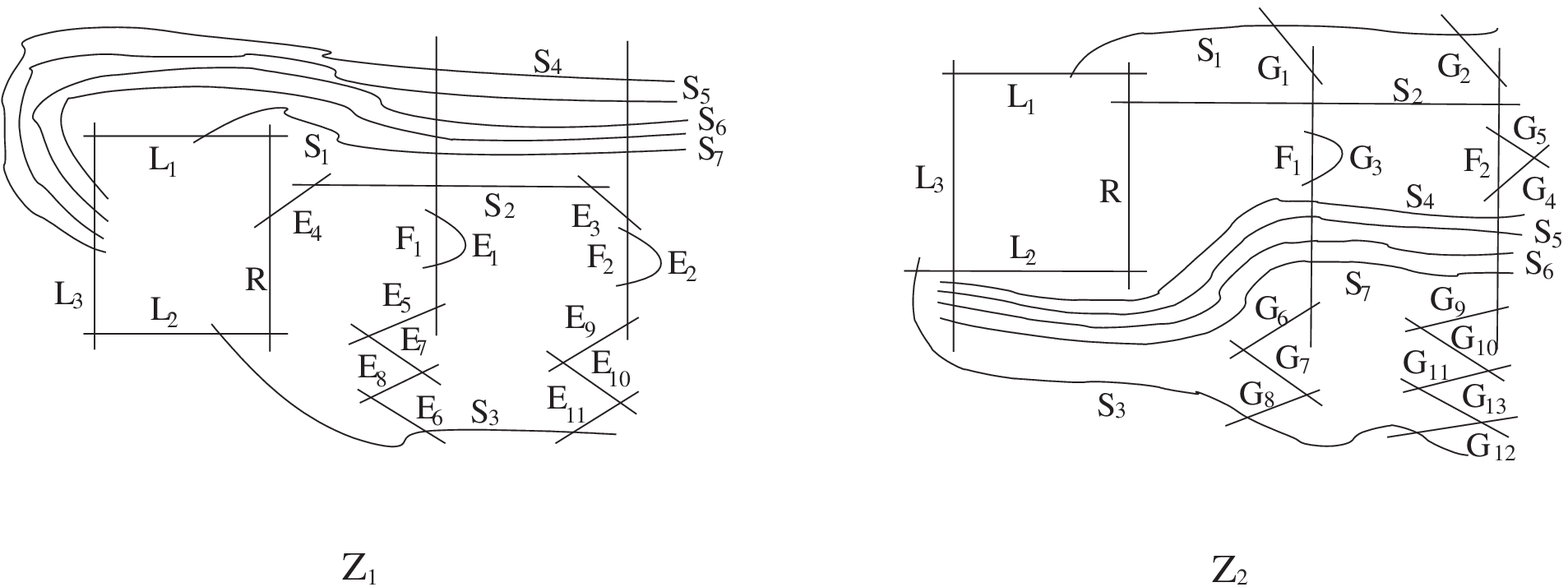}
\caption{}  \label{f8}
\end{figure}

We consider the Wahl configurations $L_1$,$R$,$L_2$,$S_3$,$E_6=[2,3,2,6,3]$, $S_2$,$F_1$,$E_5$,$E_7=[3,5,3,2]$, and $F_2$,$E_9$,$E_{10}=[6,2,2]$ in $Z_1$; and $S_1$,$L_1$,$R$, $L_2$,$S_3$,$G_{12}=[3,2,2,2,8,2]$, $F_1$,$G_6$,$G_7=[6,2,2]$, and $G_4$,$F_2$,$G_9$,$G_{10}$,$G_{11}=[2,7,2,2,3]$ in $Z_2$. Let $\sigma_i \colon Z_i \to X_i$ be the contraction of the Wahl configurations. Then, one can prove using usual techniques (as in Prop. \ref{obstr2}) that there are no local-to-global obstructions to deform $X_i$. For both cases, we can check that any $\Q$-Gorenstein smoothing produces simply connected surfaces of general type with $K^2=1$ and $p_g=0$. Let $X_0^i$ be the minimal resolution of the singularity $[2,3,2,6,3]$ for $X_1$ and $[3,2,2,2,8,2]$ for $X_2$. We consider W-surfaces $X_0^i$ which keep the corresponding exceptional divisors. Again the general fiber is the minimal resolution of stable surfaces living in two different boundary divisors of $\overline{M}$ labeled by the singularities $[2,3,2,6,3]$ and $[3,2,2,2,8,2]$. \footnote{Although in general a Wahl singularity could label two distinct divisors in $\overline{M}$.}

Let $\Gamma_t^i$ be the divisor in the $3$-fold corresponding to either the $(-6)$-curve for $i=1$ or the $(-8)$-curve for $i=2$. Let $E_{6,t}$ be the divisor corresponding to $E_6$ in $X_0^1$, and let $G_{12,t}$ be the divisor corresponding to $G_{12}$ in $X_0^2$. We now proceed to apply flips to the W-surface $X_0^i$ to find some nicer model for these deformations. We apply $3$ W-flips of the type Prop. \ref{freqFlip} for each $i$. For $X_0^1$ we flip $E_8$, $E_6$, $E_{11}$ in that order, and for $X_0^2$ we do it with $G_8$, $G_{13}$, $G_{12}$. In both cases we obtain smooth W-surfaces $Y_0^i$. We have that $$ \Gamma_0^1= F_1 + F_2 + S_3 \ \ \ \Gamma_0^2= F_1 + F_2 + S_3 \ \ \ E_{6,0}=S_2 \ \ \ G_{12,0}=G_4.$$

We now blow-down the $13$ curves $E_3$, $E_4$, $S_2$, $S_1$, $L_1$, $S_3$, $L_2$, $E_1$, $E_2$, $S_4$, $S_5$, $S_6$, $S_7$ in $Y_0^1$. Similarly, we blow-down the $14$ curves $G_5$, $G_4$, $G_3$, $S_4$, $S_5$, $S_6$, $S_7$, $G_1$, $G_2$, $S_1$, $L_1$, $S_2$, $S_3$, and $L_2$. The configuration of these exceptional curves (induced by the corresponding divisors) in $X_t^i$ is shown in Figure \ref{f9}. Notice that when we blow them down, in both cases we obtain for $\Gamma_t^i$ a sextic in $\P^2$ with $8$ nodes and one tacnode such that two nodes and the tacnode are colinear via the (transversal) line $R$. For $Y_0^i$ we obtain two nodal cubics intersecting at two tacnodes and $5$ other points, and a line $R$ so that the two tacnodes intersections, and one more (nodal) intersection are colinear by $R$. Thus, in both cases, an irreducible sextic plus a line degenerate to these two nodal cubics plus a line.

\begin{figure}[htbp]
\includegraphics[width=12cm]{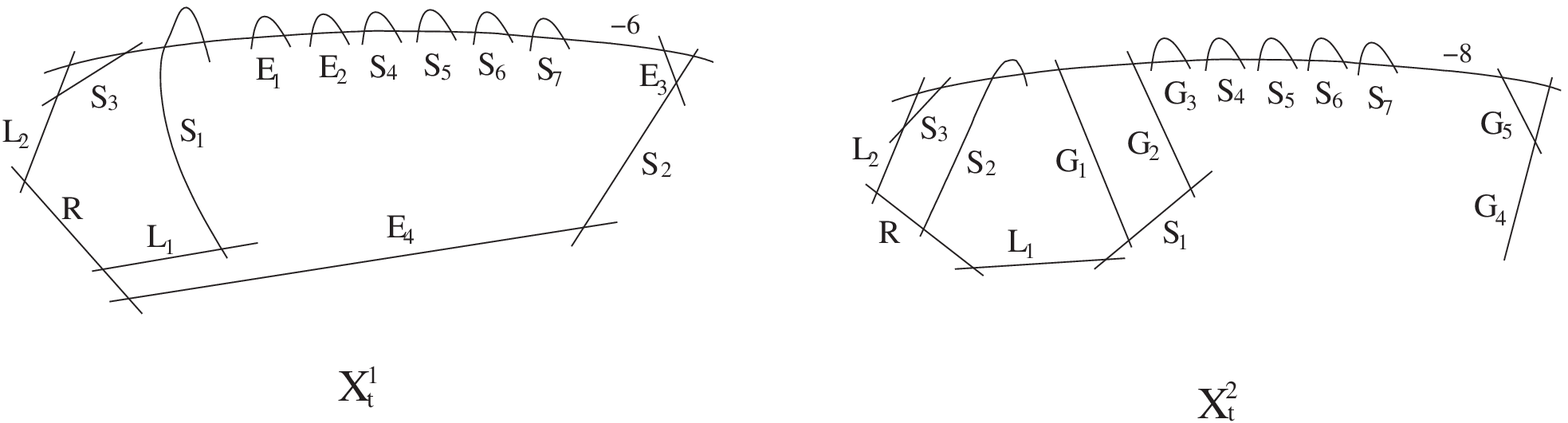}
\caption{}  \label{f9}
\end{figure}

\subsection{A maximally singular stable surface via degenerations} In this example we show how to obtain a maximally degenerated surface which has eight Wahl singularities (this is the maximum possible for an unobstructed surface), from a simple configuration of two nodal cubics and a line. That configuration ends up representing a stable surface with one Wahl singularity $\frac{1}{24^2}(1,119)$. The corresponding smoothing is simply connected with $K^2=1$ and $p_g=0$. We present this example the other way around, starting with the maximally degenerated surface, and then prove that a partial $\Q$-Gorenstein smoothing keeping only that Wahl singularity gives the simple septic above.

We start with seven lines in very special position. The lines $L_1$,...,$L_6$ form a complete quadrilateral, and $L_7$ is the line passing through $L_2 \cap L_4$ and $L_3 \cap L_6$. This is in Figure \ref{f10}. We consider the pencil generated by $L_1,L_2,L_4$ and $L_3,L_5,L_6$ which gives an elliptic fibration with singular fibers: $2I_5+2I_1$. The line $L_7$ is a particular triple section (see Figure \ref{f10}); the curves $T_1,\ldots,T_4$ come from triple points as shown.

\begin{figure}[htbp]
\includegraphics[width=9cm]{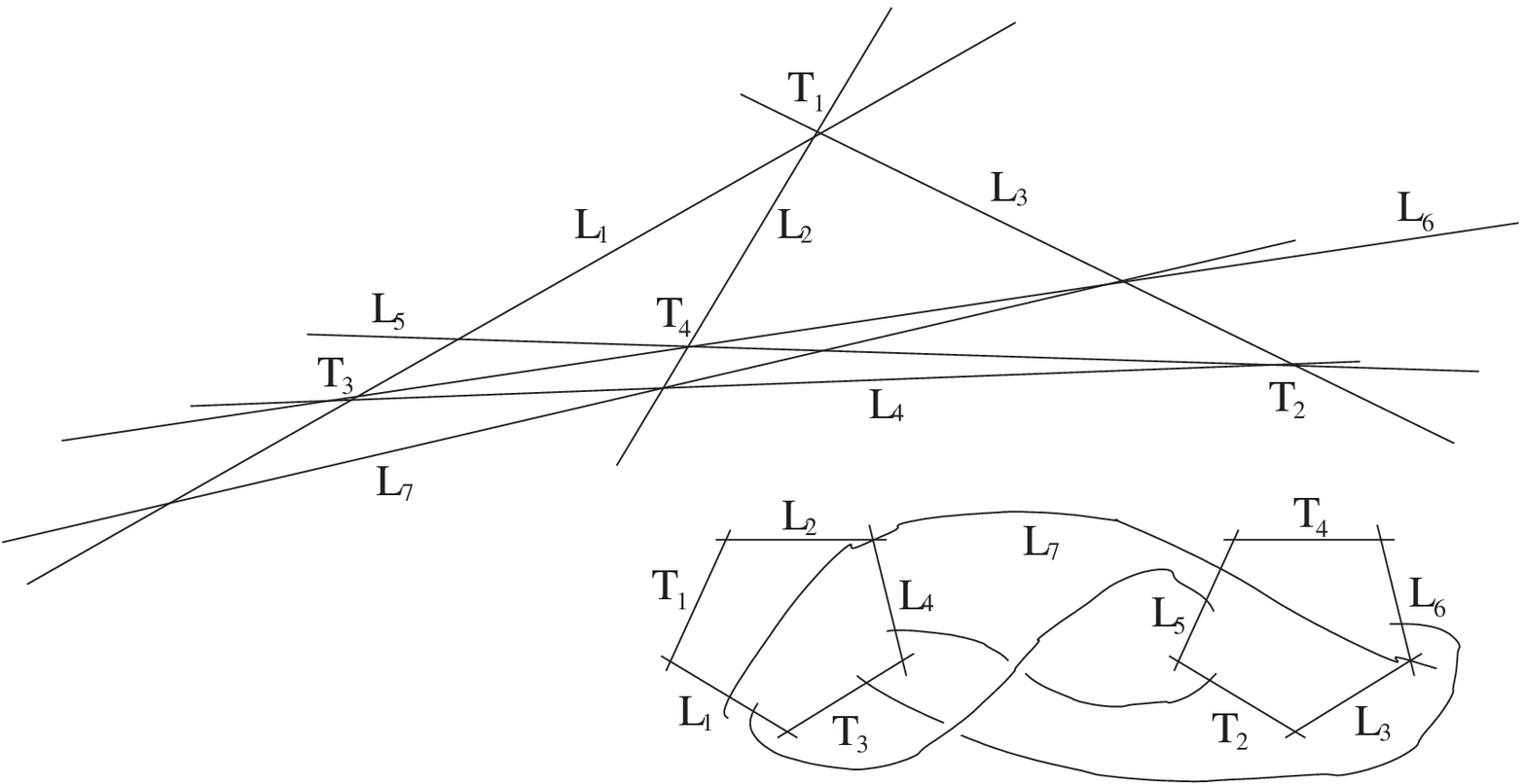}
\caption{}  \label{f10}
\end{figure}

We now blow up $26$ times obtaining a surface $Z'$. The $26$ exceptional divisors are denoted by $E_1,\ldots,E_{26}$ and again the subindex indicates the order of the composition of blow-ups. This is in Figure \ref{f11}. A nonsingular surface $Z$ is obtained by blowing down the sections $S_2$ and $S_3$ in Figure \ref{f11}. In $Z$ we have eight Wahl configurations: $L_1$,$L_5$,$T_4$,$L_6$,$T_3$,$E_{11}$,$E_{12}$,$E_{13}$ $=[5,7,2,2,3,2,2,2]$, $T_2$,$E_{19}$,$E_{20}$,$E_{21}$,$E_{22}$ $=[8,2,2,2,2]$, $L_4$,$S_1$,$E_{15}$,$E_{16}$,$E_{17}$ $=[8,2,2,2,2]$, $T_1$,$E_8$ $=[5,2]$, $L_2$,$E_7$ $=[5,2]$, $E_1=[4]$, $L_3=[4]$, and $L_7=[4]$. Let $\sigma \colon Z \to X$ be the contraction of these eight configurations.

\begin{figure}[htbp]
\includegraphics[width=10cm]{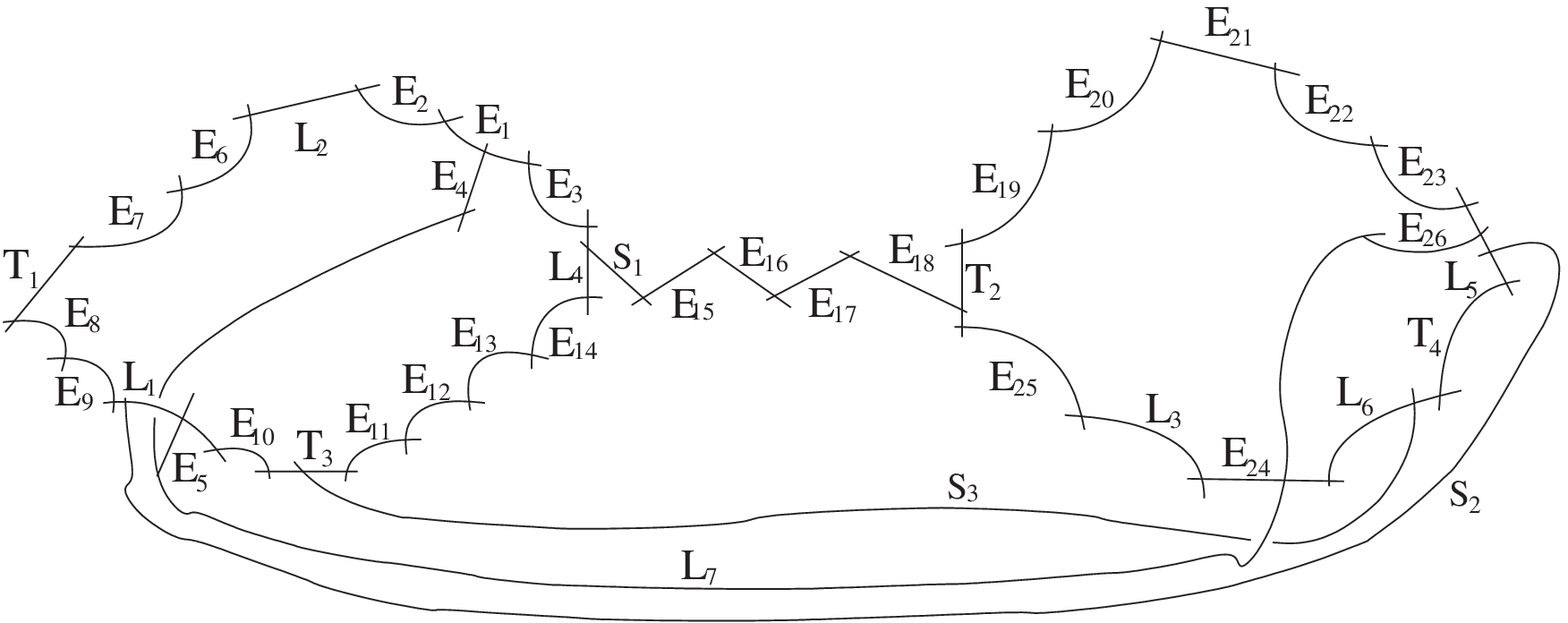}
\caption{}  \label{f11}
\end{figure}

A sketch of the proof for no local-to-global obstructions for $X$ goes as follows. We start with the minimal elliptic fibration $Y \to \P^1$ with singular fibers $2I_5+2I_1$. Then $H^2(Y,T_Y(-\log(I_5+I_5)))=0$ by Prop. \ref{obstr1}. We now add-delete $(-1)$-curves keeping this $H^2=0$. Notice that $L_7$ can be added as a $(-1)$-curve in some stage of the blow-ups (changing the given order). Thus this follows the usual procedure. Notice that $K_X^2=-26+2+25=1$.

To see that the canonical class of $K_X$ is nef, we use the trick in \cite{LP07}. We look at the resolution $f \colon Z' \to X$ and write down $f^*(K_X)$ in a $\Q$-numerically effective way. Then by intersecting with the curves in the support, one check it is nef. The two key points in the trick are (in our case): (1) to verify that the sums of the discrepancies of $L_6$ and $T_3$, and of $L_1$ and $L_5$ are smaller or equal than $-1$, (2) to verify that the discrepancies of $T_1$, $L_2$, $L_4$, $T_3$, $L_1$, $T_2$, $L_5$, $T_4$, $L_6$, and $L_3$ (the curves in the original $I_5$ fibers) have discrepancies smaller or equal than $-\frac{1}{2}$. Both (1) and (2) are true in our case. Therefore, the general fiber of a $\Q$-Gorenstein smoothing of $X$ is a surface of general type with $K^2=1$, $p_g=0$, and simply connected (for this, same strategy as in \cite{LP07}).

We now consider a minimal resolution $X' \to X$ of the singularity $[5,7,2,2,3,2,2,2]$, and a W-surface $X'$ keeping the configuration $[5,7,2,2,3,2,2,2]$. Let $\Gamma_{t,1}$, $\Gamma_{t,2}$, and $\Gamma_{t,3}$ be the divisors in the $3$-fold of the curves $L_1$, $L_5$, and $L_6$ in $X'$ respectively. We will perform seven W-flips of the type Prop. \ref{freqFlip} through the curves (and in that order): $E_9$, $E_7$, $E_2$, $E_{26}$, $E_{23}$, $E_{18}$, and $E_5$. We obtain at this point a smooth W-surface $X'$ (by abuse of notation). Then we have $$\Gamma_{0,1}=L_1+L_7+T_1+L_2+E_1 \ \ \ \Gamma_{0,2}=L_5+T_2+L_4+L_3 \ \ \ \Gamma_{0,3}=L_6$$ which already says that if we continue blowing down from $X'$ to arrive into the seven special lines, then the configuration $[5,7,2,2,3,2,2,2]$ in the general fiber blows down to two nodal cubics $\Gamma_{t,1}$ and $\Gamma_{t,2}$, and a line $\Gamma_{t,3}$. The cubics intersect transversally at nine points and three of them are collinear via $\Gamma_{t,3}$. Thus this reducible septic gives a surface with the Wahl singularity $[5,7,2,2,3,2,2,2]$, and degenerates maximally to a rigid configuration of seven lines which gives a surface with eight Wahl singularities. Finally, we observe that the surface with the one Wahl singularity $[5,7,2,2,3,2,2,2]$ specializes to the example in \cite[Sect.6]{PSU13}.


\vspace{0.3cm}

{\small Facultad de Matem\'aticas,

Pontificia Universidad
Cat\'olica de Chile,

Santiago, Chile.}

\end{document}